\documentclass[10pt]{amsart}
\usepackage[utf8]{inputenc}

\usepackage{eucal}

\usepackage{amsmath,amssymb,amsthm}
\usepackage[utf8]{inputenc}
\usepackage{url}

\usepackage[shortlabels]{enumitem}
  \setitemize[1]{leftmargin=2em}
  \setenumerate[1]{leftmargin=*}

\usepackage[colorlinks=true,linktocpage=true,citecolor=blue]{hyperref}

\newtheorem{thm}{Theorem}[section]
\newtheorem{lemma}[thm]{Lemma}
\newtheorem{propn}[thm]{Proposition}
\newtheorem{cor}[thm]{Corollary}

\newtheorem*{thm*}{Theorem}
\newtheorem*{conjecture*}{Conjecture}
\newtheorem*{prethm*}{Pretheorem}
\theoremstyle{definition}
\newtheorem{defn}[thm]{Definition}
\newtheorem{ex}[thm]{Example}

\newtheorem{remark}[thm]{Remark}

\theoremstyle{remark}

\usepackage[alphabetic]{amsrefs}
\author{Rune Haugseng}

\address{University of Copenhagen, Copenhagen, Denmark}
\email{haugseng@math.ku.dk}
\urladdr{http://sites.google.com/site/runehaugseng/}

\date{\today}

\makeatletter
\def\amsbb{\use@mathgroup \M@U \symAMSb}
\makeatother
\usepackage{mbboard}
\renewcommand{\mathbb}[1]{\amsbb{#1}}
\newcommand{\simp}{\bbDelta}
\newcommand{\blank}{\text{\textendash}}

\newcommand{\isoto}{\xrightarrow{\sim}}

\newcommand{\IFF}{if and only if}

\newcommand{\catname}[1]{\ensuremath{\text{\textup{#1}}}}
\newcommand{\txt}[1]{\ensuremath{\text{\textup{#1}}}}
\newcommand{\Set}{\catname{Set}}

\newcommand{\Cat}{\catname{Cat}}

\newcommand{\Fun}{\txt{Fun}}

\newcommand{\Map}{\txt{Map}}

\newcommand{\op}{\txt{op}}
\newcommand{\icat}{$\infty$-category}
\newcommand{\icats}{$\infty$-categories}

\newcommand{\xto}[1]{\xrightarrow{#1}}

\newcommand{\from}{\leftarrow}

\usepackage{tikz}
\usetikzlibrary{matrix,arrows}
\usepackage{tikz-cd}

\newcommand{\csquare}[8]{ %
\[ %
\begin{tikzpicture} %
\matrix (m) [matrix of math nodes,row sep=3em,column sep=2.5em,text height=1.5ex,text depth=0.25ex] %
{ #1 \pgfmatrixnextcell #2 \\ %
  #3 \pgfmatrixnextcell #4 \\ }; %
\path[->,font=\footnotesize] %
(m-1-1) edge node[auto] {$#5$} (m-1-2)%
(m-1-1) edge node[left] {$#6$} (m-2-1)%
(m-1-2) edge node[auto] {$#7$} (m-2-2)%
(m-2-1) edge node[below] {$#8$} (m-2-2);%
\end{tikzpicture}%
\]%
}

\newcommand{\smallcsquare}[8]{ %
\[ %
\begin{tikzpicture} %
\matrix (m) [matrix of math nodes,row sep=1.5em,column sep=1.25em,text height=1.5ex,text depth=0.25ex] %
{ #1 \pgfmatrixnextcell #2 \\ %
  #3 \pgfmatrixnextcell #4 \\ }; %
\path[->,font=\footnotesize] %
(m-1-1) edge node[auto] {$#5$} (m-1-2)%
(m-1-1) edge node[left] {$#6$} (m-2-1)%
(m-1-2) edge node[auto] {$#7$} (m-2-2)%
(m-2-1) edge node[below] {$#8$} (m-2-2);%
\end{tikzpicture}%
\]%
}

\newcommand{\nolabelcsquare}[4]{\csquare{#1}{#2}{#3}{#4}{}{}{}{}}
\newcommand{\nolabelsmallcsquare}[4]{\smallcsquare{#1}{#2}{#3}{#4}{}{}{}{}}

\newcommand{\smallopctriangle}[6]{ %
\[ %
\begin{tikzpicture} %
\matrix (m) [matrix of math nodes,row sep=1.5em,column sep=1em,text height=1.5ex,text depth=0.25ex] %
{  #1 \pgfmatrixnextcell \pgfmatrixnextcell #2 \\ %
  \pgfmatrixnextcell #3 \pgfmatrixnextcell \\ %
}; %
\path[->,font=\footnotesize] %
(m-1-1) edge node[above] {$#4$} (m-1-3)%
(m-1-1) edge node[below left] {$#5$} (m-2-2)%
(m-1-3) edge node[below right] {$#6$} (m-2-2);%
\end{tikzpicture}%
\]%
}

\newcommand{\id}{\txt{id}}

\DeclareMathOperator{\colimP}{colim}
\newcommand{\colim}{\mathop{\colimP}}

\DeclareMathOperator{\ob}{ob}
\title{On the Equivalence between $\bbTheta_{n}$-Spaces and Iterated Segal Spaces}

\newcommand{\bbT}{\bbTheta}
\newcommand{\bbTn}{\bbT_{n}}
\newcommand{\bbTni}{\bbT_{n,\mathrm{i}}}
\newcommand{\bbTnop}{\bbTn^{\op}}
\newcommand{\bbTniop}{\bbTni^{\op}}

\newcommand{\Di}{\simp_{\mathrm{i}}}
\newcommand{\Dni}{\Di^{n}}

\newcommand{\PSeg}{\mathcal{P}_{\txt{Seg}}}

\newcommand{\PrSeg}{\mathcal{P}_{\txt{rSeg}}}

\newcommand{\PR}{\mathcal{P}_{\txt{r}}}

\newcommand{\bbG}{\mathbb{G}}
\newcommand{\bbGn}{\bbG_{n}}
\newcommand{\bbGnop}{\bbGn^{\op}}
\newcommand{\bbGnI}{\bbG_{n/I}}

\newcommand{\Act}{\txt{Act}}

\begin{document}
\begin{abstract}
  We give a new proof of the equivalence between two of the main
  models for $(\infty,n)$-categories, namely the $n$-fold Segal spaces
  of Barwick and the $\bbTn$-spaces of Rezk, by proving that these are
  algebras for the same monad on the \icat{} of $n$-globular
  spaces. The proof works for a broad class of \icats{} that includes
  all $\infty$-topoi.
\end{abstract}
\maketitle

\section{Introduction}
$(\infty,n)$-categories are a homotopical version of
$n$-categories. This means that they have $i$-morphisms between
$(i-1)$-morphisms for $i = 1,\ldots,n$ and also homotopies between
$n$-morphisms, homotopies of homotopies, etc.\ (in other words,
invertible $i$-morphisms for $i > n$), with composition of
$i$-morphisms only associative up to a coherent choice of higher
homotopies. There are now a number of good models for
$(\infty,n)$-categories; the two that have seen the most use so
far are \emph{$n$-fold Segal spaces} and
\emph{$\bbTn$-spaces}. Iterated Segal spaces were first defined in
Barwick's thesis \cite{BarwickThesis}, building on Rezk's work on
\emph{Segal spaces} \cite{RezkCSS}, and were later generalized by
Lurie \cite{LurieGoodwillie}*{\S 1} to the setting of $\infty$-topoi;
they are presheaves of spaces on the category $\simp^{n}$ satisfying
iteratively defined ``Segal conditions'' and constancy
conditions. $\bbTn$-spaces, which were introduced by
Rezk~\cite{RezkThetaN} (no doubt influenced by Joyal's unpublished
work on $\bbTn$-sets and Berger's description of $n$-fold loop spaces
\cite{BergerWreath}), are similarly presheaves of spaces on categories
$\bbTn$ that satisfy certain ``Segal conditions''; in this paper we
consider their natural generalization to $\infty$-topoi, which we will
refer to as \emph{Segal $\bbTn$-objects} for clarity.

In \cite{BarwickSchommerPriesUnicity}, Barwick and Schommer-Pries give
axioms that characterize the homotopy theory of
$(\infty,n)$-categories. They also prove that these axioms are
satisfied in the case of $n$-fold Segal spaces and $\bbTn$-spaces,
which implies that these two models are equivalent. Another
comparison, which relates the two models directly in the setting of
model categories, has been given more recently by Bergner and Rezk
\cite{BergnerRezk2}.

The goal of this short paper is to give a new, conceptual proof of
this equivalence: we will show that both models are the \icats{} of
algebras for a monad on the \icat{} of $n$-globular spaces
(i.e.\ presheaves of spaces on the $n$-globular category,
cf.\ Definition~\ref{defn:bbGn}), and that these two monads are
equivalent. This also brings out the relation between
$(\infty,n)$-categories and $n$-categories: strict $n$-categories are
the algebras for the analogous monad on the category of $n$-globular
\emph{sets}.

Our proof only makes use of formal properties of the \icat{} of spaces
that hold for all $\infty$-topoi, so we obtain a comparison between
iterated Segal objects and Segal $\bbT_{n}$-objects in any
$\infty$-topos $\mathcal{X}$. In fact, our comparison works in
sufficient generality to allow us to conclude iteratively
that Segal $\bbT_{n_{1}+\cdots+n_{k}}$-objects in
$\mathcal{X}$ are equivalent to Segal
$\bbT_{n_{1}} \times \cdots \times \bbT_{n_{k}}$-objects that are
\emph{reduced} (i.e.\ satisfy certain constancy conditions). 

In this paper we focus on the ``algebraic'' theory of
$(\infty,n)$-categories, i.e.\ we do not invert the appropriate class
of ``fully faithful and essentially surjective morphisms''. However,
this localization is given for both $n$-fold Segal spaces and
$\bbTn$-spaces by restricting to subcategories of \emph{complete}
objects, and our equivalence is easily seen to restrict to an
equivalence between these subcategories.  

\subsection{Notation}
This paper is written in the language of \icats{}, and we reuse some
of the terminology and notation of \cite{HTT,HA} without
comment. Moreover, we use the $\infty$-categorical terminology
uniformly, even when the \icats{} in question are ordinary categories.
In particular, for us a \emph{cofinal} functor
$F \colon \mathcal{C} \to \mathcal{D}$ between ordinary categories is
a functor that is cofinal in the $\infty$-categorical sense (also
called \emph{homotopy cofinal}). Similarly, we will speak of
(co)Cartesian fibrations between categories rather than Grothendieck
(op)fibrations.

If $\mathcal{C}$ and $\mathcal{X}$ are \icats{} we will write
$\mathcal{P}(\mathcal{C}; \mathcal{X})$ for the \icat{}
$\Fun(\mathcal{C}^{\op}, \mathcal{X})$ of presheaves on $\mathcal{C}$
valued in $\mathcal{X}$.


\section{$\bbTn$-Objects and Segal Conditions}\label{sec:Tn}
In this section we will define our main objects of study in this
paper: \emph{reduced Segal $\bbTn$-objects}, which are certain
presheaves on categories $\bbTn$, whose definition we will now recall;
these categories were originally introduced by Joyal, but here we make
use of the inductive reformulation of the definition due to
Berger~\cite{BergerWreath}*{Definition 3.1}.
\begin{defn}
  The category $\bbTn$ is defined inductively as follows: First set
  $\bbT_{0} := *$. Then define $\bbT_{n}$ to be the category with
  objects $[i](I_{1},\ldots, I_{i})$ with $[i] \in \simp$ and
  $I_{p} \in \bbT_{n-1}$; a morphism
  $[i](I_{1},\ldots, I_{i}) \to [j](J_{1},\ldots,J_{j})$ is given by a
  morphism $\phi \colon [i] \to [j]$ in $\simp$ and morphisms
  $\psi_{pq} \colon I_{p}\to J_{q}$ in $\bbT_{n-1}$ where
  $0 < p \leq i$ and $\phi(p-1) < q \leq \phi(p)$. The composite of
  $(\phi, \psi_{pq}) \colon [i](I_{1},\ldots,I_{i}) \to
  [j](J_{1},\ldots,J_{j})$
  and
  $(\phi', \psi'_{qr} \colon [j](J_{1},\ldots,J_{j}) \to
  [k](K_{1},\ldots,K_{k})$
  is $(\phi' \circ \phi, \psi''_{pr})$ where
  $\psi''_{pr} := \psi'_{qr} \circ \psi_{pq}$ where $q$ is the unique
  index with $\phi(p-1)<q\leq \phi(p)$ such that
  $\phi'(q-1)< r \leq \phi'(q)$. If $\mathcal{X}$ is an \icat{}, we
  will refer to presheaves $\bbTn^{\op} \to \mathcal{X}$ as
  \emph{$\bbTn$-objects} in $\mathcal{X}$.
\end{defn}

There is a useful factorization system on $\bbTn$ given by the inert
and active morphisms, in the following sense:
\begin{defn}
  A morphism $\phi \colon [n] \to [m]$ in $\simp$ is
  \emph{inert} if it is the inclusion of a subinterval in $[m]$,
  i.e.\ $\phi(i) = \phi(0)+i$ for all $i$, and \emph{active} if it
  preserves the endpoints, i.e.\ $\phi(0) = 0$ and $\phi(n) = m$. We
  then inductively say a morphism $(\phi, \psi_{ij})$ in $\bbTn$ is \emph{inert} if $\phi$ is
  inert in $\simp$ and each $\psi_{ij}$ is inert in $\bbT_{n-1}$, and
  \emph{active} if $\phi$ is active in $\simp$ and each $\psi_{ij}$ is
  active in $\bbT_{n-1}$. We write $\bbTni$ for the subcategory of
  $\bbTn$ containing only the inert maps and $i_{n} \colon \bbTni \to
  \bbTn$ for the inclusion.
\end{defn}
Every morphism in $\bbTn$ can be factored as an active morphism
followed by an inert morphism --- moreover, as objects of $\bbTn$ have
no non-trivial automorphisms, this factorization is \emph{strictly}
unique. This factorization system seems to have been first constructed
by Berger --- it is a special case of \cite{BergerNerve}*{Lemma 1.11} (where
the inert maps are called \emph{immersions} and the active ones
\emph{covers}). It is also a special case of
\cite{WeberFamilial}*{Proposition 4.20} and of
\cite{BarwickOpCat}*{Lemma 8.3}; moreover, using the inductive
definition of $\bbTn$ it is easy to check directly by hand.

The objects of $\bbTn$ can be thought of as $n$-dimensional pasting
diagrams for compositions in $n$-categories. We now wish to define the
appropriate \emph{Segal conditions} for $\bbTn$-objects that make
their values at such a pasting diagram decompose appropriately as a
limit of the values at the basic $i$-morphisms ($i =
0,\ldots,n$).
These were originally specified by Rezk~\cite{RezkThetaN}, but we will
use an alternative formulation influenced by the work of Barwick on
operator categories \cite{BarwickOpCat}; this is also a special case of
the general version of Segal conditions considered in
\cite{WeberFamilial}. The definition requires introducing some notation:
\begin{defn}\label{defn:bbGn}
  We define objects $C_{i} \in \bbTn$ for $i = 0,\ldots,n$ by
  $C_{0} := [0]()$ and $C_{i} = [1](C_{i-1})$ for $i > 0$. (For
  $n = 0$, we let $C_{0}$ denote the unique object of $\bbT_{0} = *$.)
  Let $\bbGn$, the \emph{$n$-globular category}, be the full
  subcategory of $\bbTni$ containing the objects
  $C_{0}, \ldots, C_{n}$; we write $\gamma_{n}$ for the inclusion
  $\bbGn \hookrightarrow \bbTni$. We can informally depict the
  category $\bbGn$ as
  \[ C_{0} \rightrightarrows C_{1} \rightrightarrows \cdots
  \rightrightarrows  C_{n}.\]
  We refer to the object $C_{k}$ as the \emph{$k$-cell}.
  Given $I \in \bbTn$, we will write $\bbGnI$ for the category $\bbGn
  \times_{\bbTni} (\bbTni)_{/I}$, and refer to its objects as the
  \emph{cells} of $I$.
\end{defn}

\begin{defn}
  Suppose $\mathcal{X}$ is a presentable \icat{}. A presheaf $F \colon
  \bbTnop \to \mathcal{X}$ is a \emph{Segal $\bbTn$-object} if its
  restriction $F|_{\bbTniop}$ is the right Kan extension along
  $\gamma_{n}$ of its restriction to $\bbGnop$ --- in other words, for
  $I$ in $\bbTn$ the natural map $F(I) \to \lim_{C \in \bbGnI^{op}}
  F(C)$ is an equivalence. We write $\PSeg(\bbTn; \mathcal{X})$ for
  the full subcategory of $\mathcal{P}(\bbTn; \mathcal{X})$ spanned by
  the Segal $\bbTn$-objects, and $\PSeg(\bbTni; \mathcal{X})$ for the
  analogous subcategory of $\mathcal{P}(\bbTni; \mathcal{X})$; these
  are accessible localizations of $\mathcal{P}(\bbTn; \mathcal{X})$
  and $\mathcal{P}(\bbTni; \mathcal{X})$, respectively.
\end{defn}

This is equivalent to the inductive Segal condition given, for example, in
\cite{RezkThetaN}:
\begin{propn}\label{propn:indSegcond}
  $F \in \mathcal{P}(\bbTn; \mathcal{X})$ is a Segal presheaf \IFF{}
  the following conditions hold:
  \begin{enumerate}[(1)]
  \item For every object $I = [i](I_{1},\ldots,I_{i})$ ($i \neq 0$),
    the natural map
    \[ F(I) \to F([1](I_{1})) \times_{F(C_{0})} \cdots
    \times_{F(C_{0})} F([1](I_{i})\]
    is an equivalence.
  \item The presheaf $F([1](\blank)) \colon \bbT_{n-1}^{\op} \to
    \mathcal{X}$ is a Segal $\bbT_{n-1}$-object.
  \end{enumerate}
\end{propn}


It is convenient for us to prove Proposition~\ref{propn:indSegcond}
using a result about general limit decompositions for Segal objects, which we turn to first.
\begin{defn}
  Suppose $f \colon I \to J$ is an active morphism in $\bbTn$. For
  $\alpha \colon C \to I$ in $\bbGnI$, let $C \xto{f_{\alpha}}
  J_{\alpha} \xto{i_{\alpha}} J$ be the (unique) active-inert
  factorization of $f \circ \alpha \colon C \to J$. Given a 
  factorization of $\alpha$ as $C \xto{\xi} C' \xto{\alpha'} I$ with
  $\xi$ inert, the composite $C \to
  C' \to J_{\alpha'}$ has an active-inert factorization $C \to X \to
  J_{\alpha'}$. Since this also gives an active-inert factorization of
  $C \to J_{\alpha'} \to J$ we see that $X = J_{\alpha}$, and so $\xi$ determines an inert
  morphism $J_{\alpha} \to J_{\alpha'}$. We thus get a functor
  $\bbGnI \to \Cat$ by sending $\alpha$ to $\bbG_{n/J_{\alpha}}$
  and a morphism in $\bbGnI$ to the functor given by composition with
  the associated inert morphism $J_{\alpha} \to J_{\alpha'}$.
  Let $\bbG_{n/f} \to \bbGnI$ denote the corresponding coCartesian
  fibration. Composition with the inert morphisms $J_{\alpha} \to J$
  gives a functor $\bbG_{n/f} \to \bbG_{n/J}$.
\end{defn}

\begin{propn}\label{propn:Gnfeq}
  For any active morphism $f \colon I \to J$ in $\bbTn$, the functor
  $\bbG_{n/f} \to \bbG_{n/J}$ is cofinal.
\end{propn}

Before we prove this, we make the following simple observation:
\begin{lemma}\label{lem:Cartcof}
  Suppose $p \colon \mathcal{E} \to \mathcal{B}$ is a Cartesian
  fibration. Then $p$ is cofinal \IFF{} the fibres $\mathcal{E}_{b}$
  are weakly contractible for all $b \in \mathcal{B}$.
\end{lemma}
\begin{proof}
  By \cite{HTT}*{Theorem 4.1.3.1}, the functor $p$ is cofinal \IFF{}
  the \icats{}
  $\mathcal{E}_{b/} := \mathcal{E} \times_{\mathcal{B}}
  \mathcal{B}_{b/}$
  are weakly contractible for all $b \in\mathcal{B}$. If $p$ is
  Cartesian, for every $b \in \mathcal{B}$ the functor
  $\mathcal{E}_{b} \to \mathcal{E}_{b/}$ is coinitial (i.e.\ the op'ed
  functor is cofinal): for an object
  $\epsilon = (e \in \mathcal{E}, f \colon b \to p(e))$ in
  $\mathcal{E}_{b/}$, the \icat{} $(\mathcal{E}_{b})_{/\epsilon}$ has
  a terminal object, given by a Cartesian morphism over $f$ with
  target $e$. This functor is therefore in particular a weak homotopy
  equivalence by \cite{HTT}*{Lemma 4.1.1.3(3)}.
\end{proof}

\begin{proof}[Proof of Proposition~\ref{propn:Gnfeq}]
  The projection $\bbG_{n/f} \to \bbG_{n/J}$ is a Cartesian fibration,
  so by Lemma~\ref{lem:Cartcof} it suffices to show that for $\gamma
  \colon C \to J$, the fibre $(\bbG_{n/f})_{\gamma}$ is weakly
  contractible; we will prove this by induction on $n$. The category
  $(\bbG_{n/f})_{\gamma}$ consists of diagrams
  \[
  \begin{tikzcd}
    {} & C' \arrow{r}{\alpha} \arrow{d}{f_{\alpha}} & I \arrow{d}{f} \\
    C \arrow{r}{\gamma_{\alpha}}  \arrow[bend right=40]{rr}{\gamma} & J_{\alpha} \arrow{r}{i_{\alpha}} & J,
  \end{tikzcd}
  \]
  where $f_{\alpha}$ is active and $i_{\alpha}$ is inert. Since inert
  maps are monomorphisms in $\bbTn$, we may identify
  this with the full subcategory of $\bbG_{n/I}$ spanned by those
  cells $\alpha \colon C' \to I$ such that $\gamma$ factors through
  $i_{\alpha}$.

  First consider the case where $n = 1$. Then $\gamma$ is a map $[a]
  \to J$ where $a$ is either $0$ or $1$. If $a = 1$, then there is a
  \emph{unique} cell $\alpha \colon [1] \to I$ such that $\gamma$
  factors through $i_{\alpha}$, namely that where $f(\alpha(0)) \leq
  \gamma(0)$ and $f(\alpha(1)) \geq \gamma(1)$. Thus
  $(\bbG_{1/f})_{\gamma} = *$. On the
  other hand, if $a = 0$, then $(\bbG_{1/f})_{\gamma}$ is the full
  subcategory of $\bbG_{1/I}$ spanned by the inert maps $\alpha \colon
  [0] \to I$ such that $f(\alpha(0)) = \gamma(0)$ and $\alpha \colon
  [1] \to I$ such that $f(\alpha(0)) \leq \gamma(0)$ and $f(\alpha(1))
  \geq \gamma(0)$. The nerve of this category is an iterated pushout of
  $\Delta^{1}$'s along inclusions $\Delta^{0} \hookrightarrow
  \Delta^{1}$, and so is weakly contractible, as required.

  Now suppose the result is known for $n-1$. The cell $C$ is either $[0]()$ or
  $[1](\tilde{C})$ where $\tilde{C}$ is a cell in $\bbT_{n-1}$. If $C
  = [1](\tilde{C})$ then the underlying diagram in $\simp$ is
  unique and of the
  form
  \[
  \begin{tikzcd}
    {} & {[1]} \arrow{r}{\alpha_{0}} \arrow{d} & {[i]} \arrow{d}{f_{0}} \\
    {[1]} \arrow{r}  \arrow[bend right=40]{rr}{\gamma_{0}} & {[l]} \arrow{r} & {[j]}.
  \end{tikzcd}
  \]
  Thus $C'$ must also be of the form
  $[1](\tilde{C}')$. Let $\tilde{J} := J_{\gamma_{0}(1)}$ where $J =
  [j](J_{1},\ldots,J_{j})$, let $\tilde{\gamma}$ denote the map
  $\tilde{C} \to \tilde{J}$ induced by $\gamma$, set
$\tilde{I} := I_{\alpha_{0}(1)}$ where $I
  = [i](I_{1}, \ldots, I_{i})$, and take $\tilde{f}$ to be the induced map
  $\tilde{I} \to \tilde{J}$. Then the category $(\bbG_{n/f})_{\gamma}$
  can be identified with $(\bbG_{n-1/\tilde{f}})_{\tilde{\gamma}}$,
  which is weakly contractible by assumption.

  If $C = [0]()$, consider the projection $(\bbG_{n/f})_{\gamma} \to
  (\bbG_{1/f_{0}})_{\gamma_{0}}$ given by taking the underlying maps
  in $\simp$. This is a Cartesian fibration, so using
  Lemma~\ref{lem:Cartcof} again, it suffices to show that the fibres
  $(\bbG_{n/f})_{\gamma,\Xi}$ at $\Xi \in
  (\bbG_{1/f_{0}})_{\gamma_{0}}$ are weakly contractible. The object
  $\Xi$ is a diagram
  \[
  \begin{tikzcd}
    {} & {[a]} \arrow{r}{\alpha_{0}} \arrow{d} & {[i]} \arrow{d}{f_{0}} \\
    {[0]} \arrow{r}  \arrow[bend right=40]{rr}{\gamma_{0}} & {[l]} \arrow{r} & {[j]}.
  \end{tikzcd}
  \]
  where $a = 0$ or $1$. If $a = 0$ then the fibre is $*$, and if $a = 1$ it may be identified
  with $\bbG_{n-1/I_{\alpha_{0}(1)}}$, which is again weakly
  contractible by assumption.
\end{proof}

Applying Proposition~\ref{propn:Gnfeq} to the unique (active) map $f
  \colon J \to [0]()$ we get as a special case:
\begin{cor}\label{cor:GnIwc}
  For each $I$ in $\bbTn$, the category $\bbG_{n/I}$ is weakly contractible.\qed
\end{cor}

\begin{cor}\label{cor:actSeg}
  Suppose $F \in \mathcal{P}(\bbTn; \mathcal{X})$ is a Segal
  object. Then for any active morphism $f \colon I \to J$, the natural
  map
  \[ F(J) \to \lim_{\alpha \in \bbGnI^{op}} F(J_{\alpha})\]
  is an equivalence.
\end{cor}
\begin{proof}
  Using the Segal conditions for $J_{\alpha}$ we have 
  \[ \lim_{\alpha \in \bbGnI^{op}} F(J_{\alpha}) \simeq \lim_{\alpha
    \in \bbGnI^{op}} \lim_{C \to J_{\alpha} \in \bbG_{n/J_{\alpha}}^{\op}}
  F(C).\] By \cite{enrbimod}*{Corollary 5.7} we can rewrite this limit
  as $\lim_{C\in \bbG_{n/f}^{\op}} F(C)$, and by Proposition~\ref{propn:Gnfeq} this
  limit is equivalent to $\lim_{C \to J \in \bbG_{n/J}^{\op}} F(C)$, which
  we know by the Segal condition for $J$ is equivalent to $F(J)$.
\end{proof}

\begin{proof}[Proof of Proposition~\ref{propn:indSegcond}]
  First suppose $F \colon \bbTnop \to \mathcal{X}$ is a Segal
  object. Given $I = [i](I_{1},\ldots,I_{i})$ ($i \neq 0$), set
  $\tilde{I} := [i](C_{n-1},\ldots,C_{n-1})$ and let $f \colon
  \tilde{I} \to I$ denote the (active) map given by the identity $[i]
  \to [i]$ and the unique active maps $C_{n-1}\to I_{p}$. If $\Lambda$
  denotes the full subcategory of $\bbG_{n/\tilde{I}}$ containing the $i$
  maps $C_{n} \to \tilde{I}$ and the $i+1$ maps $C_{0} \to
  \tilde{I}$, then the inclusion $\Lambda \hookrightarrow \bbG_{n/\tilde{I}}$
  is cofinal. Together with Corollary~\ref{cor:actSeg} this gives an
  equivalence
  \[ F(I) \isoto \lim_{\alpha \in \bbG_{n/\tilde{I}}^{\op}}
  F(I_{\alpha}) \isoto F([1](I_{1})) \times_{F(C_{0})} \cdots
  \times_{F(C_{0})} F([1](I_{i})),\]
  which is condition (1). Condition (2) holds since the functor
  $\bbG_{n-1/J} \to \bbG_{n/[1](J)}$ induced by $[1](\blank) \colon
  \bbT_{n-1} \to \bbT_{n}$ is cofinal.

  Now suppose conditions (1) and (2) hold for $F$; we then wish to
  show that $F(I) \to \lim_{C \to I \in \bbG_{n/I}^{\op}} F(C)$ is an
  equivalence for
  any $I = [i](I_{1},\ldots,I_{i}) \in \bbTn$. With $f \colon \tilde{I} \to I$ as above,
  Proposition~\ref{propn:Gnfeq} implies that it suffices to prove that
  the natural map $F(I) \to \lim_{\bbG_{n/f}^{\op}} F$ is an
  equivalence. Using \cite{enrbimod}*{Corollary 5.7} we can rewrite
  this as an iterated limit $\lim_{\alpha \in \bbG_{n/\tilde{I}}}
  \lim_{\bbG_{n/I_{\alpha}}} F$. But now using the same cofinal
  functors as above, we can rewrite this again as
  \[ \left(\lim_{C \to I_{1} \in \bbG_{n-1/I_{1}}} F([1](C))\right)
  \times_{F(C_{0})} \cdots \times_{F(C_{0})} \left(\lim_{C \to I_{i} \in
    \bbG_{n-1/I_{i}}} F([1](C)) \right), \]
  which is equivalent to $F(I)$ by (1) and (2).
\end{proof}

For the \icat{} $\mathcal{S}$ of spaces, $\PSeg(\bbTn; \mathcal{S})$
is the \icat{} underlying Rezk's model category of $\bbTn$-spaces from
\cite{RezkThetaN}. More generally, if $\mathcal{X}$ is, say, an
$\infty$-topos, the \icat{} $\PSeg(\bbTn; \mathcal{X})$ gives the
(algebraic) \icat{} of internal $(\infty,n)$-categories in
$\mathcal{X}$. We would like to be able to iterate this definition, so
that we get a good definition of Segal $\bbT_{m}$-objects in
$\PSeg(\bbTn; \mathcal{X})$. Just as in Barwick's definition of
$n$-fold Segal spaces, this requires forcing some of the images to be
constant; to formalize this notion, it is convenient to introduce the
following technical definition:
\begin{defn}\label{defn:sble}
  A \emph{presentable \icat{} with good constants} is a pair
  $(\mathcal{X}, \mathcal{U})$ consisting of an \icat{} $\mathcal{X}$
  together with a full subcategory $\mathcal{U}$ satisfying the
  following requirements:
  \begin{enumerate}[(a)]
  \item $\mathcal{X}$ and $\mathcal{U}$ are both presentable.
  \item The inclusion $\mathcal{U} \hookrightarrow \mathcal{X}$
    preserves all limits and colimits (and hence, by the adjoint
    functor theorem, has both a left and a right adjoint).
  \item Coproducts in $\mathcal{U}$ are disjoint, i.e.\ for any two
    objects $U, U'\in \mathcal{U}$, the commutative square
    \nolabelsmallcsquare{\emptyset}{U}{U'}{U \amalg U'}
    is Cartesian.
  \item Coproducts over $\mathcal{U}$ are universal, i.e.\ for any
    morphism $f \colon X \to U$ in $\mathcal{X}$ with $U \in
    \mathcal{U}$, the functor $f^{*} \colon \mathcal{X}_{/U} \to
    \mathcal{X}_{/X}$, given by pullback along $f$, preserves the
    initial object and arbitrary coproducts.
  \end{enumerate}
\end{defn}

\begin{ex}
  If $\mathcal{X}$ is an $\infty$-topos, then $(\mathcal{X},
  \mathcal{X})$ is a presentable \icat{} with good constants by
  \cite{HTT}*{Theorem 6.1.0.6}.
\end{ex}

\begin{remark}
  Since we are requiring pullbacks over $\mathcal{U}$ to preserve
  \emph{all} coproducts in $\mathcal{X}$, not just coproducts in
  $\mathcal{U}$, a \emph{distributor} in the sense of
  Lurie~\cite{LurieGoodwillie}*{Definition 1.2.1} is not necessarily
  a presentable \icat{} with good constants. However, the key examples
  --- $\infty$-topoi and iterated $\bbTn$-objects in $\infty$-topoi
  --- are both distributors \emph{and} presentable \icats{} with good constants.
\end{remark}

\begin{defn}
  Suppose $(\mathcal{X}, \mathcal{U})$ is a presentable \icat{} with
  good constants. We say a presheaf $X \in \mathcal{P}(\bbGn;
  \mathcal{X})$ is \emph{reduced} if $X(C_{i})$ is in $\mathcal{U}$
  for all $i< n$; we write $\PR(\bbGn; \mathcal{X},
  \mathcal{U})$ for the full subcategory of $\mathcal{P}(\bbGn;
  \mathcal{X})$ spanned by the reduced objects. A Segal object $X$
  in $\PSeg(\bbTni; \mathcal{X})$ or $\PSeg(\bbTn;
  \mathcal{X})$ is then called \emph{reduced} if $X|_{\bbGn^{\op}}$ is
  reduced; we
  write $\PrSeg(\bbTn; \mathcal{X}, \mathcal{U})$ and $\PrSeg(\bbTni;
  \mathcal{X}, \mathcal{U})$ for the full subcategories of
  $\mathcal{P}(\bbTn; \mathcal{X})$ and $\mathcal{P}(\bbTni;
  \mathcal{X})$, respectively, spanned by the reduced Segal objects.
\end{defn}

\begin{propn}\label{propn:PrSeggoodconst}
  Suppose $(\mathcal{X}, \mathcal{U})$ is a presentable \icat{} with
  good constants.
  \begin{enumerate}[(i)]
  \item The \icat{} $\PrSeg(\bbTn; \mathcal{X}, \mathcal{U})$ is
    presentable, and the inclusion \[\PrSeg(\bbTn; \mathcal{X},
    \mathcal{U}) \hookrightarrow \mathcal{P}(\bbTn; \mathcal{X})\]
    admits a left adjoint $L_{n}$.
  \item The functor $c^{*} \colon \mathcal{U} \to \mathcal{P}(\bbTn;
    \mathcal{X})$ that takes an object in $\mathcal{U}$ to the
    constant presheaf with that value is fully faithful and takes
    values in $\PrSeg(\bbTn; \mathcal{X}, \mathcal{U})$.
  \item The pair $(\PrSeg(\bbTn; \mathcal{X}, \mathcal{U}),
    \mathcal{U})$, with $\mathcal{U}$ viewed as the full subcategory
    of constant presheaves, is a presentable \icat{} with good
    constants.
  \end{enumerate}
\end{propn}

Before we give the proof of this proposition, we need some technical lemmas:
\begin{lemma}\label{lem:pbcoprod}
  Let $(\mathcal{X}, \mathcal{U})$ be a presentable \icat{} with good
  constants. Suppose given maps of sets $f\colon A \to B$ and $g
  \colon C \to B$, objects $X_{a} \in \mathcal{X}$ for $a \in A$,
  $Y_{c} \in \mathcal{X}$ for $c \in C$, $U_{b} \in \mathcal{U}$ for
  $b \in B$, and morphisms $\phi_{a}\colon X_{a} \to U_{f(a)}$ and
  $\psi_{c} \colon Y_{c} \to U_{g(c)}$ in $\mathcal{X}$ for all $a \in
  A$ and $c \in C$. Then the natural map
  \[ \coprod_{(a,b,c) \in A \times_{B} C} X_{a} \times_{U_{b}} Y_{c}
  \to \left(\coprod_{a \in A} X_{a}\right) \times_{\left(\coprod_{b
        \in B} U_{b}\right)} \left(\coprod_{c \in C} Y_{c}\right)\]
  is an equivalence in $\mathcal{X}$.
\end{lemma}
\begin{proof}
  For $b \in B$, let $A_{b}$ and $C_{b}$ denote the
  fibres of $f$ and $g$ at $b$. Then 
  condition (d) in Definition~\ref{defn:sble} gives equivalences
\[ \coprod_{(a,b,c) \in A \times_{B} C} X_{a} \times_{U_{b}} Y_{c}
\simeq \coprod_{b \in B} \coprod_{(a,c) \in A_{b} \times C_{b}} X_{a}
\times_{U_{b}} Y_{c} \simeq \coprod_{b \in B} \left(\coprod_{a \in A_{b}} X_{a}\right) \times_{ U_{b}}
  \left(\coprod_{c \in C_{b}} Y_{c}\right),\]
\[ \left(\coprod_{a \in A} X_{a}\right) \times_{\left(\coprod_{b
        \in B} U_{b}\right)} \left(\coprod_{c \in C} Y_{c}\right)
  \simeq \coprod_{b',b'' \in B} \left(\coprod_{a \in A_{b'}} X_{a}\right) \times_{\left(\coprod_{b
        \in B} U_{b}\right)} \left(\coprod_{c \in C_{b''}}
    Y_{c}\right).\]
  Let $\tilde{X}_{b}:= \amalg_{a \in A_{b}} X_{a}$ and $\tilde{Y}_{b} := \amalg_{c
    \in C_{b}} Y_{c}$, then it remains to show that 
  \[ \tilde{X}_{b'} \times_{\left(\coprod_{b \in B}
      U_{b}\right)} \tilde{Y}_{b''} \simeq
  \begin{cases}
    \emptyset, & b' \neq b'',\\
    \tilde{X}_{b'} \times_{U_{b'}} \tilde{Y}_{b'}, & b' = b''.
  \end{cases}
  \]
  Since $\tilde{X}_{b'} \times_{\left(\coprod_{b \in B} U_{b}\right)} \tilde{Y}_{b''}
  \simeq \tilde{X}_{b'} \times_{U_{b'}} U_{b'} \times_{\left(\coprod_{b \in B}
      U_{b}\right)} U_{b''} \times_{U_{b''}} \tilde{Y}_{b''}$ and
  pullbacks over
  objects in $\mathcal{U}$ preserve the initial object, it is enough to show
  that
  \[ U_{b'} \times_{\left(\coprod_{b \in B}
      U_{b}\right)} U_{b''} \simeq
  \begin{cases}
    \emptyset, & b' \neq b'',\\
    U_{b'}, & b' = b''
  \end{cases}
  \]
  To see this
  we observe that, setting $V :=
  \coprod_{b \neq b'} U_{b}$, for $b' \neq b''$ we have
  \[ U_{b'} \times_{\left(\coprod_{b \in B}
      U_{b}\right)} U_{b''} \simeq U_{b'} \times_{U_{b'} \amalg V} V
  \times_{V} U_{b''} \simeq \emptyset \times_{V} U_{b''} \simeq \emptyset,\]
  using that coproducts in $\mathcal{U}$ are disjoint and pullbacks in
  $\mathcal{U}$ preserve the initial object. For $i = j$ we
  have
  \[ U_{b'} \simeq U_{b'} \times_{U_{b'} \amalg V} (U_{b'} \amalg V)
  \simeq (U_{b'} \times_{U_{b'} \amalg V} U_{b'}) \amalg (U_{b'}
  \times_{U_{b'} \amalg V} V) \simeq U_{b'} \times_{U_{b'} \amalg V}
  U_{b'}.\qedhere\]
\end{proof}

\begin{lemma}\label{lem:coprodSeg}
  Given a set $S$ and objects $Y_{i} \in \PrSeg(\bbTn; \mathcal{X}, \mathcal{U})$ for
  $i \in S$, the coproduct $Y := \coprod_{i \in S} Y_{i}$ in
  $\mathcal{P}(\bbTn; \mathcal{X})$ is a reduced Segal
  $\bbTn$-object.
\end{lemma}
\begin{proof}
  Since $\mathcal{U}$ is closed under colimits in
  $\mathcal{X}$, the object $Y$ is reduced. Applying
  Lemma~\ref{lem:pbcoprod} we see that for $I =
  [i](I_{1},\ldots,I_{i})$ the natural map
  \[ Y(I) \to Y([1](I_{1})) \times_{Y(C_{0})} \cdots \times_{Y(C_{0})}
  Y([1](I_{i})) \]
  is an equivalence. By Proposition~\ref{propn:indSegcond} this
  implies by induction on $n$ that $Y$ is a Segal
  object.
\end{proof}

\begin{proof}[Proof of Proposition~\ref{propn:PrSeggoodconst}]
  The \icat{} $\PrSeg(\bbTn; \mathcal{X}, \mathcal{U})$ fits in a
  commutative diagram
  \[
  \begin{tikzcd}
    \PrSeg(\bbTn; \mathcal{X}, \mathcal{U}) \arrow{d} \arrow{r}& \PSeg(\bbTn;
    \mathcal{X}) \arrow{d} \\
    \PR(\bbTn; \mathcal{X}, \mathcal{U}) \arrow{d} \arrow{r}& \mathcal{P}(\bbTn;
    \mathcal{X}) \arrow{d} \\
    \mathcal{P}(\bbG_{n-1}; \mathcal{U}) \arrow{r} & \mathcal{P}(\bbG_{n-1}; \mathcal{X}).
  \end{tikzcd}
  \]
  where both squares are Cartesian. Moreover, the bottom horizontal
  and the two right vertical functors are right adjoints between
  presentable \icats{}. By \cite{HTT}*{Theorem 5.5.3.18} limits in the
  \icat{} $\Pr^{\mathrm{R}}$ of presentable \icats{} and right
  adjoints are computed in that of large \icats{}, hence all \icats{} in this
  diagram are presentable and all functors are right adjoints. This
  proves (i).

  Since $\bbTn$ is weakly contractible (as it has a terminal object)
  the image of the constant presheaf functor $c^{*} \colon \mathcal{U}
  \to \mathcal{P}(\bbTn; \mathcal{U}) \to \mathcal{P}(\bbTn;
  \mathcal{X})$ is fully faithful. Constant presheaves on objects in
  $\mathcal{U}$ satisfy the Segal condition by Corollary~\ref{cor:GnIwc},
  so this functor factors through $\PrSeg(\bbTn; \mathcal{X},
  \mathcal{U})$, which gives (ii).

  For (iii), we already know conditions (a) and (c) in
  Definition~\ref{defn:sble}. Limits in $\PrSeg(\bbTn; \mathcal{X}, \mathcal{U})$
  are computed in $\mathcal{P}(\bbTn; \mathcal{X})$, i.e.\ objectwise,
  and colimits are given by the localizations of the corresponding
  colimits in $\mathcal{P}(\bbTn; \mathcal{X})$; since constant
  presheaves on objects in $\mathcal{U}$ are already local, this
  implies condition (b). It remains to check condition (d), 
  i.e.\ given maps $Y_{i}\to c^{*}U$ for $i \in S$ we need to show that the natural
  map
  \[ \coprod_{i} X \times_{c^{*}U} Y_{i} \to X \times_{c^{*}U}
  \coprod_{i}Y_{i}\] is an equivalence. By Lemma~\ref{lem:coprodSeg}
  these coproducts can be computed in $\mathcal{P}(\bbTn;
  \mathcal{X})$, so it suffices to show that for $I \in \bbTn$ we have that
  \[ \coprod_{i} X(I) \times_{U} Y_{i}(I) \to X(I) \times_{U}
  \coprod_{i}Y_{i}(I)\]
  is an equivalence, which is true since $U$ is in $\mathcal{U}$.
\end{proof}

\begin{defn}
  For $(\mathcal{X}, \mathcal{U})$ a presentable \icat{} with good
  constants, we write $\PrSeg(\bbTn \times \bbT_{m}; \mathcal{X},
  \mathcal{U})$ for the full subcategory of $\mathcal{P}(\bbTn \times
  \bbT_{m}; \mathcal{X})$ corresponding to $\PrSeg(\bbTn;
  \PrSeg(\bbT_{m}; \mathcal{X}, \mathcal{U}),
  \mathcal{U})$. Similarly, we define
  $\PrSeg(\bbT_{n_{1}} \times \cdots \times \bbT_{n_{k}}; \mathcal{X},
  \mathcal{U})$ and $\PrSeg(\bbT_{n_{1},\mathrm{i}} \times \cdots
  \times \bbT_{n_{k},\mathrm{i}}; \mathcal{X}, \mathcal{U})$ by induction.
\end{defn}

\begin{ex}
  The \icat{} $\PrSeg(\simp^{n}; \mathcal{S})$ is the \icat{} of
  Barwick's \emph{$n$-fold Segal spaces} \cite{BarwickThesis}. More
  generally, $\PrSeg(\simp^{n}; \mathcal{X}, \mathcal{U})$ gives
  Lurie's $n$-fold $\mathcal{U}$-Segal spaces from
  \cite{LurieGoodwillie}.
\end{ex}

\section{The Free Reduced Segal $\bbTn$-Object Monad}\label{sec:monad}
Our goal in this section is to show that the \icat{}
$\PrSeg(\bbTn; \mathcal{X}, \mathcal{U})$ is the \icat{} of algebras
for a monad on $\PR(\bbGn; \mathcal{X}, \mathcal{U})$, and to
understand this monad explicitly. This is closely related to the
arguments used by Berger in the proof of \cite{BergerNerve}*{Theorem
  1.12}. Before we state our precise result, we must introduce some
notation:
\begin{defn}
  For $I \in \bbTn$, let $\Act(I)$ denote the set of active morphisms
  $I \to J$ in $\bbTn$. A morphism $f \colon I' \to I$ determines a
  map of sets $f^{*} \colon \Act(I) \to \Act(I')$ by taking $\phi
  \colon I \to J$ to the active morphism $\phi' \colon I' \to J'$ that
  gives the (unique) active-inert factorization of $I' \to I \to J$. Since this
  factorization is unique, it is easy to see that this determines a
  functor $\Act \colon \bbTn^{\op} \to \Set$.
\end{defn}

\begin{defn}
  Define $\iota_{n} \colon \bbT_{n-1} \to \bbTn$ inductively by taking
  $\iota_{1} \colon * = \bbT_{0} \to \bbT_{1} = \simp$ to be the
  inclusion of $[0]$ and setting \[\iota_{n}([m](I_{1},\ldots,I_{m})) =
  [m](\iota_{n-1}(I_{1}),\ldots,\iota_{n-1}(I_{m})).\] Notice that
  $\iota_{n}$ is fully faithful. We write $\iota_{k}^{n} := \iota_{n}
  \circ \cdots \circ \iota_{k+1} \colon \bbT_{k} \to \bbTn$.
\end{defn}

\begin{propn}\label{propn:Tnmonad} Let $(\mathcal{X}, \mathcal{U})$ be
  a presentable \icat{} with good constants.
  \begin{enumerate}[(i)]
  \item The functor $i_{n}^{*} \colon \PrSeg(\bbTn; \mathcal{X},
    \mathcal{U}) \to \PrSeg(\bbTni; \mathcal{X}, \mathcal{U})$ has a
    left adjoint $F_{n}$.
  \item The adjunction $F_{n} \dashv i_{n}^{*}$ is monadic.
  \item The monad $T_{n} := i_{n}^{*}F_{n}$ on $\PrSeg(\bbTni;
    \mathcal{X}, \mathcal{U})$ satisfies
    \[ T_{n}X(I) \simeq \coprod_{I \to J \in \Act(I)} X(J).\] In
    particular,
    \[ T_{n}X(C_{k}) \simeq \coprod_{J \in \bbT_{k}} X(\iota_{k}^{n}J).\]
  \end{enumerate}
\end{propn}

The proof relies on a simple description of the left Kan extension
functor $i_{n,!}$, which we prove first:
\begin{lemma}\label{lem:inLKan}
  The functor $i_{n,!} \colon \mathcal{P}(\bbTni; \mathcal{X}) \to
    \mathcal{P}(\bbTn; \mathcal{X})$ can be described explicitly as
    \[ i_{n,!}F(I) \simeq \coprod_{I \to J \in \Act(I)} F(J). \] In
    particular, $i_{n,!}F(C_{k}) \simeq \coprod_{I \in \ob \bbT_{k}}
    F(\iota_{k}^{n}(I))$.
\end{lemma}
\begin{proof}
  We first show that the inclusion
  $\Act(I) \to (\bbTniop)_{/I} := \bbTniop \times_{\bbTnop}
  (\bbTnop)_{/I}$
  is cofinal. By \cite{HTT}*{Theorem 4.1.3.1} this is equivalent to
  the category $(\Act(I))_{/X}$ being weakly contractible for each
  $X = (J, f \colon I \to J)$ in $(\bbTni)_{I/}$. But this category 
  consists of active-inert factorizations of $f$, and so is
  contractible as this factorization is unique. Hence the left Kan extension
  $i_{n,!}F$ is indeed given by \[i_{n,!}F(I) \simeq
  \colim_{(I \to J) \in (\bbTniop)_{/I}} F(J) \simeq \coprod_{I \to J
    \in \Act(I)} F(J).\]
  If $I = C_{k}$ then the only objects of $\bbTn$ that admit an active
  map from $C_{k}$ are those in the image of the fully faithful
  functor $\iota_{k}^{n}\colon  \bbT_{k} \to \bbT_{n}$ (and these active maps
  are unique), which gives the expression for $i_{n,!}F(C_{i})$.
\end{proof}

We need one more observation:
\begin{lemma}\label{lem:actSeg}
  The functor $\Act \colon \bbTnop \to \Set$ is a Segal $\bbTn$-object.
\end{lemma}
\begin{proof}
  We prove this by induction on $n$, using the criterion of
  Proposition~\ref{propn:indSegcond}. For $I = [i](I_{1},\ldots,I_{i})
  \in \bbTn$, the definition of active morphisms in $\bbTn$
  immediately implies that $\Act(I) \cong \Act([1](I_{1})) \times \cdots
  \times \Act([1](I_{i}))$ and $\Act(C_{0}) \cong *$, which gives condition
  (1). To prove (2), suppose $I = [1](J)$ for some $J \in \bbT_{n-1}$. Then
  it is immediate from the definition of active maps in $\bbTn$ that
  \[ \Act(I) \cong \coprod_{i = 0}^{\infty} \Act'(J)^{\times i}\]
  (where for clarity we write $\Act'$ for the $\bbT_{n-1}$-version
  of $\Act$). By assumption we have $\Act'(J) \cong \lim_{C \to J \in \bbG_{n-1/J}^{\op}}
  \Act'(C)$, hence as limits commute and coproducts in $\Set$ commute
  with connected limits, we have isomorphisms
  \[ \begin{split}\Act(I) & \cong \coprod_{i=0}^{\infty} \left(\lim_{C \to J \in \bbG_{n-1/J}^{\op}}
      \Act'(C)\right)^{\times i} \cong  \lim_{C \to J \in \bbG_{n-1/J}^{\op}} \left
      (\coprod_{i = 0}^{\infty} \Act'(C)^{\times i}\right) \\ & \cong \lim_{C \to J \in \bbG_{n-1/J}^{\op}} \Act([1](C)),
  \end{split}
  \]
  which is condition (2).
\end{proof}

\begin{proof}[Proof of Proposition~\ref{propn:Tnmonad}]
  Let $L_{n}$ denote the localization functor from $\mathcal{P}(\bbTn;
  \mathcal{X})$ to $\PrSeg(\bbTn; \mathcal{X}, \mathcal{U})$; then
  $L_{n}i_{n,!}$ clearly restricts to a left adjoint to $i_{n}^{*}$,
  which gives (i).

  To see that the adjunction is monadic it suffices by
  \cite{HA}*{Theorem 4.7.4.5} to prove that $i_{n}^{*}$ detects
  equivalences and that colimits of $i_{n}^{*}$-split simplicial
  objects exist in $\PrSeg(\bbTn; \mathcal{X}, \mathcal{U})$ and are
  preserved by $i_{n}^{*}$. Since $\bbTni$ is a subcategory of $\bbTn$
  containing all the objects it is clear that $i_{n}^{*}$ detects
  equivalences. Suppose we have an $i_{n}^{*}$-split simplicial object
  $X_{\bullet}$ in $\PrSeg(\bbTn; \mathcal{X}, \mathcal{U})$,
  i.e.\ $i_{n}^{*}X_{\bullet}$ extends to a split simplicial object
  $X'_{\bullet} \colon \simp_{-\infty}^{\op} \to \PrSeg(\bbTni;
  \mathcal{X}, \mathcal{U})$ (where $\simp_{-\infty}$ is as in
  \cite{HA}*{Definition 4.7.3.1}). If we consider $X_{\bullet}$ as a
  diagram in $\mathcal{P}(\bbTn; \mathcal{X})$ with colimit $X$, then
  this colimit is preserved by $i^{*}_{n} \colon \mathcal{P}(\bbTn;
  \mathcal{X}) \to \mathcal{P}(\bbTni; \mathcal{X})$ (since this
  functor is a left adjoint). But by \cite{HA}*{Remark 4.7.3.3}, the
  diagram $X'_{\bullet}$ is a colimit diagram also when viewed as a
  diagram in $\mathcal{P}(\bbTni; \mathcal{X})$, so $i^{*}_{n}X \simeq
  X'_{-\infty}$. This means that $X$ is a reduced Segal $\bbTn$-object, and so
  it is also the colimit of $X_{\bullet}$ in $\PrSeg(\bbTn;
  \mathcal{X}, \mathcal{U})$, and its image in $\PrSeg(\bbTni;
  \mathcal{X}, \mathcal{U})$ is $X'_{-\infty}$, as required. This
  proves (ii).

  To prove (iii), we will show that if $X \in \PrSeg(\bbTni;
  \mathcal{X}, \mathcal{U})$ then $i_{n,!}X$ is a reduced Segal
  $\bbTn$-space, hence $F_{n}X$ is just given by the left Kan
  extension $i_{n,!}X$:

  To see that $i_{n,!}$ is reduced, we observe that for $i < n$ the
  expression for $i_{n,!}F(C_{i})$ in Lemma~\ref{lem:inLKan} is a
  coproduct of limits of objects in $\mathcal{U}$, and hence is also
  in $\mathcal{U}$ since this is closed in $\mathcal{X}$ under all
  limits and colimits.

  Now since $X$ is a Segal $\bbTni$-object we have, using
  Proposition~\ref{propn:Gnfeq},
  \[ i_{n,!}X(I) \simeq \coprod_{I \to J \in \Act(I)} X(J) \simeq
  \coprod_{I \to J \in \Act(I)} \lim_{\alpha \in \bbGnI}
  X(J_{\alpha}).\] These limits over $\bbGnI$ can be rewritten as
  iterated pullbacks over objects in $\mathcal{U}$, and by
  Lemma~\ref{lem:actSeg} we have that $\Act(I)$ is equivalent to
  $\lim_{\alpha \colon C \to I \in \bbGnI} \Act(C)$. Applying
  Lemma~\ref{lem:pbcoprod} iteratively we can then conclude that the
  natural map
  \[ \coprod_{I \to J \in \Act(I)} \lim_{\alpha \in
    \bbGnI} X(J_{\alpha}) \to \lim_{\alpha \colon C \to I \in \bbGnI}
  \coprod_{C \to J_{\alpha}} X(J_{\alpha})\]
  is an equivalence. Here the target is equivalent to $\lim_{\alpha
    \colon C \to I \in \bbGnI} i_{n,!}X(C)$, i.e.\ $i_{n,!}X$ satisfies
  the Segal condition. The expression for $F_{n}X(C_{i})$ is then
  immediate from Lemma~\ref{lem:inLKan}.
\end{proof}

\section{Comparison}\label{sec:comp}
Our goal in this section is to prove our comparison result. More
precisely, we will show:
\begin{thm}\label{thm:main} 
  Let $\tau_{1,n}\colon \simp \times \bbTn \to \bbT_{n+1}$ be the
  functor determined by sending $([n], I)$ to $[n](I, \ldots,
  I)$. Then composition with $\tau_{1,n}$ induces, for $(\mathcal{X},
  \mathcal{U})$ a presentable \icat{} with good constants, an
  equivalence
  \[\tau_{1,n}^{*} \colon \PrSeg(\bbT_{n+1}; \mathcal{X},
  \mathcal{U}) \isoto \PrSeg(\simp \times \bbTn; \mathcal{X},
  \mathcal{U}).\]
\end{thm}
Iterating this result, we get:
\begin{cor}
  Let $\tau_{k,n} \colon \simp^{k} \times \bbT_{n} \to \bbT_{n+k}$ be
  defined inductively as 
  \[ \simp^{k} \times \bbT_{n} \xto{\id_{\simp} \times \tau_{k-1,n}}
  \simp \times \bbT_{n+k-1} \xto{\tau_{1,n+k-1}} \bbT_{n+k}.\] Then
  for $(\mathcal{X}, \mathcal{U})$ any presentable \icat{} with good
  constants the functor
  \[ \tau_{k,n}^{*} \colon \PrSeg(\simp^{k} \times \bbT_{n};
  \mathcal{X}, \mathcal{U}) \to \PrSeg(\bbT_{n+k}; \mathcal{X},
  \mathcal{U})\]
  is an equivalence. \qed
\end{cor}
In particular, taking $\mathcal{X}$ to be an $\infty$-topos and $n =
0$ we get an equivalence between the \icat{} $\PrSeg(\simp^{k};
\mathcal{X})$ of $k$-fold Segal spaces in $\mathcal{X}$ and the
\icat{} $\PSeg(\bbT_{k}; \mathcal{X})$ of Segal $\bbT_{k}$-objects in
$\mathcal{X}$.

\begin{remark}
  Similarly, applying Theorem~\ref{thm:main} inductively we get for
  any sequence of positive integers $(n_{1},\ldots, n_{k})$ an
  equivalence between $\PrSeg(\bbT_{n_{1}} \times \cdots \times
  \bbT_{n_{k}}; \mathcal{X}, \mathcal{U})$ and
  $\PrSeg(\bbT_{n_{1}+\cdots+n_{k}}; \mathcal{X}, \mathcal{U})$.
\end{remark}

To prove Theorem~\ref{thm:main}, we will use the following analogue of
Proposition~\ref{propn:Tnmonad}:
\begin{propn}\label{propn:T1nmonad}\
  \begin{enumerate}[(i)]
  \item Let $i_{1,n}:= i_{1} \times i_{n} \colon \Di \times \bbTni \to
    \simp \times \bbTn$. The functor \[i_{1,n}^{*} \colon \PrSeg(\simp
    \times \bbTn; \mathcal{X}, \mathcal{U}) \to \PrSeg(\Di \times
    \bbTni; \mathcal{X}, \mathcal{U})\] has a left adjoint $F_{1,n}$.
  \item The adjunction $F_{1,n} \dashv i_{1,n}^{*}$ is monadic.
  \item The monad $T_{1,n} := i_{1,n}^{*}F_{1,n}$ on $\PrSeg(\Dni
    \times \bbTni;
    \mathcal{X}, \mathcal{U})$ satisfies
    \[ T_{1,n}X([0], C_{0}) \simeq X([0], C_{0}),\]
    \[ T_{1,n}X([1], C_{k}) \simeq \coprod_{j = 0}^{\infty}
    F_{n}\tilde{X}(C_{k}) \times_{X([0], C_{0})} \cdots\times_{X([0],
      C_{0})} F_{n}\tilde{X}(C_{k}),\] where $\tilde{X}:= X([1],
    \blank)$ and the factor $F_{n}\tilde{X}(C_{k})$ occurs $j$ times.
  \end{enumerate}
\end{propn}

For the proof we need the following observations:
\begin{lemma}\label{lem:sliceadj}
  Suppose $L : \mathcal{C} \rightleftarrows \mathcal{D} : R$ is an
  adjunction. Then for any $d \in \mathcal{D}$ there is an adjunction
  \[ L_{d} : \mathcal{C}_{/Rd} \rightleftarrows \mathcal{D}_{/d} :
  R_{d},\]
  where $L_{d}(x \to Rd)$ is the composite $Lx \to LRd \to d$ using
  the counit, and $R_{d}(y \to d)$ is $Ry \to Rd$.
\end{lemma}
\begin{proof}
  Let $\eta \colon LR \to \id$ be the counit for the adjunction. This
  determines a natural transformation $\eta_{d} \colon L_{d}R_{d} \to
  \id$, and the map \[\Map_{\mathcal{C}_{/Rd}}(x, R_{d}y) \to
  \Map_{\mathcal{D}_{/LRd}}(L_{d}x, L_{d}R_{d}y) \to
  \Map_{\mathcal{D}_{/d}}(L_{d}x, y)\] is the map on fibres at $x \to Rd$ of
  the commutative square
 \nolabelcsquare{\Map_{\mathcal{C}}(x, Ry)}{\Map_{\mathcal{D}}(Lx,
   y)}{\Map_{\mathcal{C}}(x, Rd)}{\Map_{\mathcal{D}}(Lx, d)}
 induced by $\eta$. Here both horizontal maps are equivalences, since
 $\eta$ is the counit of the adjunction $L \dashv R$, hence so is the
 map on fibres. The natural transformation $\eta_{d}$ is therefore the
 counit of an adjunction $L_{d}\dashv R_{d}$ by
 \cite{HTT}*{Proposition 5.2.2.8}.
\end{proof}

\begin{lemma}\label{lem:funG1fibre}
  If $\mathcal{C}$ is an \icat{} with finite products, then for all $x
  \in \mathcal{C}$ there is a pullback square
  \nolabelsmallcsquare{\mathcal{C}_{/x \times x}}{\Fun(\bbG_{1}^{\op},
    \mathcal{C})}{\{x\}}{\mathcal{C}}
  where the right vertical map is given by evaluation at $C_{0}$.
\end{lemma}
\begin{proof}
  Let $\Lambda$ denote the category $0 \to \infty \from 1$,
  i.e. $\{0,1\}^{\triangleright}$. Then for $x,y \in \mathcal{C}$ we
  have by \cite{spans}*{Lemma 6.4} a pullback square
  \nolabelsmallcsquare{\mathcal{C}_{/x \times y}}{\Fun(\Lambda^{\op},
    \mathcal{C})}{\{x,y\}}{\mathcal{C} \times \mathcal{C},}
  where the right vertical map is given by evaluation at $0$ and
  $1$. Now $\bbG_{1}$ can be identified with the pushout $\Lambda
  \amalg_{\{0,1\}} *$, so we also have a pullback square
  \nolabelsmallcsquare{\Fun(\bbG_{1}^{\op},
    \mathcal{C})}{\Fun(\Lambda^{\op},
    \mathcal{C})}{\mathcal{C}}{\mathcal{C} \times \mathcal{C},}
  where the bottom horizontal map is the diagonal map. Putting these
  two squares together gives the result.
\end{proof}

\begin{proof}[Proof of Proposition~\ref{propn:T1nmonad}]
  The functor $i_{1,n} \colon \Di \times \bbTni \to \simp \times
  \bbTn$ factors as the composite of inclusions $i'_{1,n} := \id
  \times i_{n} \colon \Di \times \bbTni \to \Di \times \bbTn$ and
  $i''_{1,n} := i_{1} \times \id \colon \Di \times \bbTn \to \simp
  \times \bbTn$. Here $(i''_{1,n})^{*}$ is just $i_{1}^{*}$ applied to
  the presentable \icat{} with good constants $(\PrSeg(\bbTn;
  \mathcal{X}, \mathcal{U}), \mathcal{U})$, so by
  Proposition~\ref{propn:Tnmonad} it has a left adjoint, given by
  $i_{1,!}$. To prove (i) it therefore suffices to show that
  $(i'_{1,n})^{*}$ has a left adjoint.
  
  Let $\PrSeg(\bbG_{1} \times \bbTn; \mathcal{X}, \mathcal{U})$ denote
  the full subcategory of
  $\mathcal{P}(\bbG_{1} \times \bbTn; \mathcal{X}, \mathcal{U})$
  spanned by those presheaves $X$ such that $X(C_{0})$ is a constant
  presheaf valued in $\mathcal{U}$ and $X(C_{1})$ is in
  $\PrSeg(\bbTn; \mathcal{X}, \mathcal{U})$. Then right Kan extension
  along $\gamma_{1} \times \id$ gives an equivalence between
  $\PrSeg(\bbG_{1} \times \bbTn; \mathcal{X}, \mathcal{U})$ and
  $\PrSeg(\Di \times \bbTn; \mathcal{X}, \mathcal{U})$. Similarly,
  $\PrSeg(\Di \times \bbTni; \mathcal{X}, \mathcal{U})$ is equivalent
  to the analogous full subcategory
  $\PrSeg(\bbG_{1} \times \bbTni; \mathcal{X}, \mathcal{U})$ of
  $\mathcal{P}(\bbG_{1} \times \bbTni; \mathcal{X},
  \mathcal{U})$. Under these equivalences $(i'_{1,n})^{*}$ corresponds
  to the functor $j^{*}$ where $j := \id_{\bbG_{1}} \times i_{n} \colon
  \bbG_{1} \times \bbTni \to \bbG_{1} \times \bbTn$.

  We then have a commutative triangle \smallopctriangle{\PrSeg(\bbG_{1}
    \times \bbTn; \mathcal{X}, \mathcal{U})}{\PrSeg(\bbG_{1} \times
    \bbTni; \mathcal{X}, \mathcal{U})}{\mathcal{U},}{j^{*}}{}{} 
  where the diagonal morphisms are given by evaluation at $C_{0}$,
  since this takes values in the constant presheaves on $\bbTn$ valued
  in $\mathcal{U}$.
  
  By Lemma~\ref{lem:funG1fibre} we can identify the morphism on fibres at $U \in
  \mathcal{U}$ with the functor
  \[ \PrSeg(\bbTn; \mathcal{X}, \mathcal{U})_{/U \times U} \to
  \PrSeg(\bbTni; \mathcal{X}, \mathcal{U})_{/U \times U}\]
  given by composing with $i_{n}$, where $U \times U$ denotes the
  constant presheaf with this value. Lemma~\ref{lem:sliceadj}
  therefore implies that the functor $j^{*}$ has a left adjoint on the
  fibre over each $U \in \mathcal{U}$, given by applying $F_{n}$ and
  composing with the counit map to the constant presheaf. By
  \cite{HA}*{Proposition 7.3.2.6} this implies that $(i'_{1,n})^{*}$
  has a left adjoint, provided the two diagonal functors are
  Cartesian fibrations and $j^{*}$ preserves Cartesian morphisms.

  The functor
  $\txt{ev}_{C_{0}} \colon \PrSeg(\bbG_{1} \times \bbTn; \mathcal{X},
  \mathcal{U}) \to \mathcal{U}$
  has a right adjoint $R$, taking $U \in \mathcal{U}$ to the right Kan
  extension along $\{C_{0}\} \hookrightarrow \bbG_{1}$ of the constant
  presheaf on $U$. (Thus $R(U)(C_{1})$ is the constant presheaf on
  $U \times U$.) To prove that $\txt{ev}_{C_{0}}$ is a Cartesian
  fibration we can therefore apply the criterion of
  \cite{nmorita}*{Corollary 4.52}: We must show that given
  $X \in \PrSeg(\bbG_{1} \times \bbTn; \mathcal{X}, \mathcal{U})$ and
  a morphism $f \colon U \to X(C_{0})$ in $\mathcal{U}$, if we define
  $X'$ as the pullback of $X \to R(X(C_{0}))$ along
  $R(U) \to R(X(C_{0}))$, then the morphism
  $X'(C_{0}) \to R(U)(C_{0}) \simeq U$ is an equivalence; this is
  clear since pullbacks in
  $\PrSeg(\bbG_{1} \times \bbTn; \mathcal{X}, \mathcal{U})$ are
  computed objectwise. The map $X' \to X$ is a Cartesian morphism over
  $f$ with target $X$; the same argument shows that the other functor
  is likewise a Cartesian fibration, and $j^{*}$ preserves
  Cartesian morphisms as it preserves pullbacks. This completes the
  proof of (i).

  (ii) now follows by the same argument as in the proof of
  Proposition~\ref{propn:Tnmonad}(ii), and (iii) by
  Proposition~\ref{propn:Tnmonad} and our description of the left
  adjoints to $(i'_{1,n})^{*}$ and $(i''_{1,n})^{*}$.
\end{proof}

\begin{proof}[Proof of Theorem~\ref{thm:main}]
  The functor $\tau_{1,n}^{*}$ takes reduced Segal
  $\bbT_{n+1}$-objects to reduced Segal $\simp \times \bbTn$-objects:
  condition (1) in Proposition~\ref{propn:indSegcond} implies the
  Segal condition in the $\simp$-coordinate, and condition (2) implies
  it in the $\bbT_{n}$-coordinate; since $\tau_{1,n}([0], \blank)$ is
  constant at $[0]()$, we see that $\tau_{1,n}^{*}$ also preserves
  reducedness.
  
  We then have a commutative diagram
  \[
  \begin{tikzcd}
    \PrSeg(\bbT_{n+1}; \mathcal{X}, \mathcal{U})
    \arrow{r}{\tau_{1,n}^{*}} \arrow{d}{i_{n+1}^{*}} & \PrSeg(\simp
    \times \bbTn; \mathcal{X},
    \mathcal{U}) \arrow{d}{i_{1,n}^{*}} \\
    \PrSeg(\bbT_{n+1,\mathrm{i}}; \mathcal{X}, \mathcal{U})
    \arrow{r}{\tau^{*}_{1,n,\mathrm{i}}} \arrow{d}{\gamma_{n+1}^{*}} &
    \PrSeg(\Di \times \bbTni; \mathcal{X}, \mathcal{U})
    \arrow{d}{\gamma_{1,n}^{*}} \\
    \PR(\bbG_{n+1}; \mathcal{X}, \mathcal{U}) \arrow{r}{\beta_{n}^{*}}
    & \PR(\bbG_{1} \times \bbG_{n+1}; \mathcal{X}, \mathcal{U}),
  \end{tikzcd}
  \]
  where $\gamma_{1,n} := \gamma_{1} \times \gamma_{n}$ and $\beta$ is
  the restriction of $\tau_{1,n}$ to a functor
  $\colon \bbG_{1} \times \bbG_{n} \to \bbG_{n+1}$ (this sends
  $(C_{0}, C_{i})$ to $C_{0}$ and $(C_{1},C_{i})$ to $C_{i+1}$).

  We will now prove that the functor $\tau_{1,n,\mathrm{i}}^{*}$ is
  an equivalence. Observe that the vertical morphisms in the
  bottom square are equivalences, since the Segal presheaves on
  $\bbT_{n+1,\mathrm{i}}$ and $\Di \times \bbTni$ are precisely those
  presheaves that are right Kan extensions along $\gamma_{n+1}$ and
  $\gamma_{1} \times \gamma_{n}$ of presheaves on $\bbG_{n+1}$ and
  $\bbG_{1}\times \bbG_{n}$, respectively. It therefore suffices to
  show that the functor $\beta_{n}^{*}$ is an equivalence. The reduced
  presheaves on $\bbG_{1} \times \bbG_{n+1}$ are precisely those that
  are in the image under $\beta_{n}^{*}$ of
  $\PR(\bbG_{n+1}; \mathcal{X}, \mathcal{U})$, so to see this it is
  enough to prove that $\beta_{n}^{*}$ is fully faithful. Consider the
  left adjoint
  $\beta_{n,!} \colon \mathcal{P}(\bbG_{1}\times\bbG_{n}; \mathcal{X})
  \to \mathcal{P}(\bbG_{n+1}; \mathcal{X})$,
  given by left Kan extension along
  $\beta^{\op} \colon \bbG_{1}^{\op}\times\bbG_{n}^{\op} \to
  \bbG_{n+1}^{\op}$.
  The category $(\bbG_{1}^{\op}\times\bbG_{n}^{\op})_{/C_{k}}$ has a
  terminal object for every $k = 0,\ldots,n+1$, namely
  $((C_{1}, C_{i-1}), \beta(C_{1},C_{i-1}) \xto{\id} C_{i})$ for
  $i > 0$ and $(C_{0}, C_{0}), \beta(C_{0},C_{0}) \xto{\id} C_{0})$
  for $i = 0$. Thus for
  $F \in \mathcal{P}(\bbG_{1}\times\bbG_{n}; \mathcal{X})$ we
  have \[\beta_{!}F(C_{i}) \simeq
  \begin{cases}
    F(C_{1},C_{i-1}) & i > 0 \\
    F(C_{0}, C_{0}) & i = 0.
  \end{cases}
  \]
  The counit $\beta_{!}\beta^{*}\to \id$ is therefore an equivalence,
  which implies that $\beta^{*}$ is fully faithful, as required.

  The vertical maps in the top square above are monadic right
  adjoints by Proposition~\ref{propn:Tnmonad} and
  Proposition~\ref{propn:T1nmonad}. To see that $\tau_{1,n}^{*}$ is an
  equivalence it then suffices, by \cite{HA}*{Corollary 4.7.4.16}, to
  show that for every $X \in \PR(\bbG_{n+1}; \mathcal{X}, \mathcal{U}) \simeq
  \PrSeg(\bbT_{n+1,\mathrm{i}}; \mathcal{X}, \mathcal{U})$ the unit map $X \to
  i_{n+1}^{*}F_{n+1} \simeq i_{1,n}^{*}\tau_{1,n}^{*}F_{n+1}$ induces an
  equivalence $F_{1,n}X \isoto \tau_{1,n}^{*}F_{n+1}X$, or (since
  $i_{1,n}^{*}$ detects equivalences) the induced map
  $(F_{1,n}X)(C_{k}) \to (F_{n+1}X)(C_{k})$ is an equivalence for $k =
  0,\ldots,n+1$.

  To prove this we will rewrite our expression for
  $(F_{n+1}X)(C_{k})$ from Proposition~\ref{propn:Tnmonad}, which says
  \[ F_{n+1}X(C_{k}) \simeq \coprod_{I \in \ob\bbT_{k}} \iota_{k}^{n+1,*}X(I).\]
  Let $(\ob \bbT_{k})_{i}$ denote the subset of $\ob \bbT_{k}$
  consisting of objects of the form $[i](\cdots)$. 
  
  By Proposition~\ref{propn:indSegcond} we get for every object $I =
  [i](I_{1},\ldots,I_{i})$ in $\bbT_{k}$ an equivalence 
  \[ \iota_{k}^{n+1,*}X(I) \simeq \iota_{k}^{n+1,*}X(\sigma_{k}I_{1})
  \times_{\iota_{k}^{n+1,*}X(C_{0})}
  \cdots\times_{\iota_{k}^{n+1,*}X(C_{0})}\iota_{k}^{n+1,*}X(\sigma_{k}I_{i}),\]
  where $\sigma_{k} \colon \bbT_{k-1} \to \bbT_{k}$ is the functor
  $[1](\blank)$. There is a bijection
  $(\ob \bbT_{k})_{i} \cong (\ob \bbT_{k-1})^{\times (i-1)}$, and since coproducts
  over $\mathcal{U}$ are universal we can rewrite our expression for
  $F_{n}X(C_{k})$ as
  \[ \coprod_{i = 0}^{\infty} \left(\coprod_{I_{1} \in \bbT_{k-1}}
    \iota_{k}^{n+1,*}X(\sigma_{k}I_{1})\right) \times_{X(C_{0})} \cdots \times_{X(C_{0})} \left(\coprod_{I_{i} \in \bbT_{k-1}}
    \iota_{k}^{n+1,*}X(\sigma_{k}I_{i})\right).\]
  Here, as $\iota_{k}^{n+1}\sigma_{k} =
  \sigma_{n+1}\iota_{k-1}^{n}$, we have equivalences
  \[ \coprod_{I' \in
    \bbT_{k-1}} \iota_{k}^{n+1,*}X(\sigma_{k}I') \simeq \coprod_{I' \in
    \bbT_{k-1}} \iota_{k-1}^{n,*}(\sigma_{n+1}^{*}X)(I') \simeq F_{n}(\sigma_{n+1}^{*}X)(C_{k-1}).\]
  Comparing this to the expression for $F_{1,n}$ in
  Proposition~\ref{propn:T1nmonad} then completes the proof.
\end{proof}

\begin{remark}
  Let $E^{n}$ denote the nerve of the (contractible) category with $n$
  objects and a unique morphism between any two objects, viewed as
  a Segal space. Then a Segal space is \emph{complete} if it is local
  with respect to the map $E^{1} \to E^{0}$. We can then inductively
  define a Segal $\simp \times \bbT_{n}$-space $X$ to be complete if
  $X(\blank, C_{0})$ is a complete Segal space and $X([1], \blank)$ is
  a complete Segal $\bbT_{n}$-space, where a Segal $\bbT_{n}$-space
  $Y$ is complete if $\tau_{1,n-1}^{*}Y$ is a complete $\simp \times
  \bbT_{n-1}$-space. Expanding this out, it is easy to see that it
  recovers Rezk's definition of complete $\bbT_{n}$-spaces, and that
  under our equivalence $\PrSeg(\simp^{n}) \simeq \PSeg(\bbTn)$ the
  complete $\bbTn$-spaces correspond to the complete $n$-fold Segal
  spaces as defined by Barwick.
\end{remark}

\end{document}